\documentclass{article}
\usepackage[dvips]{graphics}
\usepackage[english]{babel}
\usepackage{amsfonts}
\usepackage{amsthm}
\usepackage{amsmath}
\usepackage{amscd}
\usepackage[active]{srcltx} 
\usepackage{latexsym}
\usepackage{amssymb}
\usepackage[latin1]{inputenc}
\usepackage{color}

\numberwithin{equation}{section}

\newtheorem{thm}{Theorem}[section]
\newtheorem{prop}[thm]{Proposition}
\newtheorem{cor}[thm]{Corollary}
\newtheorem*{cor*}{Corollary}
\newtheorem{lema}[thm]{Lemma}
\newtheorem*{lema*}{Lemma}

\theoremstyle{definition}

\newtheorem*{prob*}{Problem}

\newtheorem{Def}[thm]{Definition}

\newtheorem{exa}[thm]{Example}
\newtheorem{obs}[thm]{Remark}

\newtheorem*{obs*}{Remark}
\newtheorem*{thm*}{Theorem}
\newtheorem*{prop*}{Proposition}

\newcommand{\matriz}[4]{\displaystyle\
    \left(
       \begin{array}{cc}
        {#1}&{#2}\\
        {#3}&{#4}
       \end{array}
     \right)}

\newcommand{\PI}[2]{\left\langle \,#1 , #2\, \right\rangle}

\newcommand{\K}[2]{\left[ \,#1 , #2\, \right]}

\newcommand{\ra}{\rightarrow}

\newcommand{\s}{\sigma}

\newcommand{\CC}{\mathbb{C}}

\newcommand{\St}{\mathcal{S}}

\newcommand{\N}{\mathcal{N}}

\newcommand{\M}{\mathcal{M}}

\newcommand{\F}{\mathcal{F}}
\newcommand{\G}{\mathcal{G}}
\newcommand{\HH}{\mathcal{H}}
\newcommand{\KK}{\mathcal{K}}

\newcommand{\mc}[1]{\mathcal{#1}}

\newcommand{\ol}{\overline}

\newcommand{\ort}{[\bot]}

\newcommand{\sdo}{[\dotplus]}

\newcommand{\noi}{\noindent}

\newcommand{\Skindef}{\K{ \,\, }{\,}}
\newcommand{\Skdef}{\PI{\,\,}{\,}}

\DeclareMathOperator{\Span}{span}

\DeclareMathOperator{\sgn}{sgn}

\begin{document}

\title{Duality for Frames in Krein Spaces}

\author{J.~I.~Giribet, A.~Maestripieri, and F.~Mart\'{\i}nez~Per\'{\i}a
\thanks{This work was partially supported by \textit{Consejo Nacional de Investigaciones Cient\'{\i}ficas y T\'ecnicas} CONICET PIP 0168.
The research of J. Giribet and A. Maestripieri was also partially supported by \textit{Universidad de Buenos Aires}  UBACyT 01/Q637, and \textit{Agencia Nacional de Promoci\'on Cient\'ifica y Tecnol\'ogica} PICT-2014-1776.
The research of F. Mart\'{\i}nez Per\'{\i}a was also partially supported by \textit{Universidad Nacional de La Plata}  UNLP 11X681, and \textit{Agencia Nacional de Promoci\'on Cient\'ifica y Tecnol\'ogica} PICT-2015-1505.}
}
\date{}

\maketitle

\begin{abstract}
A $J$-frame for a Krein space $\HH$ is in particular a frame for $\HH$ (in the Hilbert space sense). But it is also compatible with the indefinite inner-product of $\HH$, meaning that it determines a pair of maximal uniformly definite subspaces,
an analogue to the maximal dual pair associated to an orthonormal basis in a Krein space. This work is devoted to study duality for $J$-frames in Krein spaces. Also, tight and Parseval $J$-frames are defined and characterized.
\end{abstract}



\section{Introduction}

Filtering, prediction and smoothing are fundamental problems in signal processing. The Krein space approach has attracted increasing attention from researchers in that field, mainly because its robustness. 
Based on the Krein space estimation theory developed in \cite{HSK96}, several applications have been described in detail, including $H_\infty$ filtering, quadratic game theory, finite-memory adaptive filtering, risk-sensitive control and estimation problems. Moreover, Krein space based filtering, prediction and smoothing techniques have been well developed in the last years, see \cite{S14} and the references therein.

Frame theory is a key tool in signal and image processing, data compression and sampling theory, among other applications in engineering, applied mathematics and computer sciences. The major advantage of a frame over an orthonormal, orthogonal or Riesz basis is its redundancy: each vector admits several different reconstructions in terms of the frame coefficients. For instance, frames have shown to be useful in signal processing applications when noisy channels are involved, because a frame allows to reconstruct vectors (signals) even if some of the frame coefficients are missing (or corrupted), see \cite{BodPau,Stro,HolPau}.

A frame for a Hilbert space $(\HH,\PI{\,\,}{\,})$ is a family $\mc{F}=\{f_i\}_{i\in I}$ in $\HH$ for which there exist positive constants $0<\alpha\leq \beta$ such that
\[
  \alpha\, \|f\|^2 \leq \sum_{i\in I} |\PI{f}{f_i}|^2 \leq \beta\, \|f\|^2, 
\]
for every $f\in \HH$. Every frame $\mc{F}$ for $\HH$ has associated a (bounded) positive invertible operator $S:\HH\ra\HH$ (the so-called frame operator), which allows to reconstruct each $f\in \HH$ with the following sampling-reconstruction scheme: if $g_i:=S^{-1}f_i$ then 
\begin{equation}\label{dualidad}
  f = \sum_{i\in I}\PI{f}{f_i}\, g_i=\sum_{i\in I}\PI{f}{g_i}f_i.
\end{equation}
The family $\{S^{-1}f_i\}_{i\in I}$ is also a frame for $\HH$, and it is called the \emph{canonical dual frame} of $\mc{F}$. Given a frame $\mc{F}=\{f_i\}_{i\in I}$ for $\HH$, in general there are infinitely many frames $\mc{G}=\{g_i\}_{i\in I}$ for $\HH$ which are ``duals'' for $\F$, in the sense that $\F$ and $\G$ satisfies \eqref{dualidad}. Each of these duals frames gives an alternative sampling-reconstruction scheme. 

A frame $\mc{F}=\{f_i\}_{i\in I}$ is called a Parseval frame if it is self-dual. Thus, the associated frame operator is the identity, 
and the signals $f\in\HH$ can be decomposed as:
\begin{equation*}
  f = \sum_{i\in I}\PI{f}{f_i}\, f_i.
\end{equation*}
In this case, no computation is needed to calculate the dual frame, a desirable property for several applications, see \cite{Eldar15}. Another useful property of sampling with Parseval frames is that the energy of the samples is the same as the energy of the original signal.

Recently, there were different attempts to introduce frame theory on Krein spaces;
see \cite{EFW,GMMPM12,PW}. Along this work, we use the notion of $J$-frame proposed in \cite{GMMPM12} which was motivated by a signal processing problem, where signals with the same energy at high and low band frequencies are considered disturbances, see the discussion in \cite[Sec. 3]{GMMPM12}. With the proper Krein space structure, the disturbances are embedded into the set of neutral vectors, and signals can be classified into signals with predominantly high-frequencies  (which are positive vectors in the sense of Krein spaces) and predominantly low-frequencies (negative vectors in the sense of Krein spaces). 

A $J$-frame $\F=\{f_i\}_{i\in I}$ for a Krein space $(\HH, \K{\,\,}{\,})$ is a family of non-neutral vectors such that the subfamily $\F_+=\{f_i\}_{i\in I_+}$ of positive vectors generates a maximal uniformly positive subspace $\M_+$ of $\HH$, that is, they generate a (maximal) subspace which is itself a Hilbert space with the inner-product $\K{\,\,}{\,}$. Moreover, the correlation between the signals in $\M_+$ and the disturbances (neutral vectors) is uniformly bounded, see \cite[Sec. 3.2]{GMMPM12}. Analogously, the subfamily $\F_-=\{f_i\}_{i\in I_-}$ of negative vectors generate a maximal uniformly negative subspace $\M_-$ of $\HH$, i.e. an anti-Hilbert space with the inner-product $\K{\,\,}{\,}$. Also, the $J$-frame $\F$ allows to decompose every signal $f\in \HH$ as $f=f_+ + f_-$, where $f_\pm\in \M_\pm$, and $\F_\pm$ is a frame for the Hilbert space $(\M_\pm, \pm\K{\,\,}{\,})$.

The $J$-frame operator $S:\HH\ra\HH$ associated to a $J$-frame $\F$ is an invertible selfadjoint operator in the Krein space sense. It determines the following sampling-reconstruction scheme for $\F$: given $f\in \HH$, it can be represented as
\[
	f=\sum_{i\in I}\sigma_i\K{f}{f_i}S^{-1}f_i=\sum_{i\in I}\sigma_i\K{f}{S^{-1}f_i}f_i,
\]
where $\sigma_i=\sgn\K{f_i}{f_i}$. 

\medskip

In this work the duality for $J$-frames is studied. Given a $J$-frame $\F=\{f_i\}_{i\in I}$ for $\HH$, a family $\G=\{g_i\}_{i\in I}$ in $\HH$ is a dual family for $\F$ if $\sgn\K{g_i}{g_i}=\sgn\K{f_i}{f_i}=\sigma_i$ and 
\begin{equation}\label{dual2}
	f=\sum_{i\in I}\sigma_i\K{f}{f_i}g_i=\sum_{i\in I}\sigma_i\K{f}{g_i}f_i,
\end{equation}
for every $f\in \HH$. There are infinitely many dual families for a given $J$-frame, and most of them are frames (in the Hilbert space sense) but not $J$-frames. However, if $S$ is the $J$-frame operator associated to $\F$, the family $\F'=\{S^{-1}f_i\}_{i\in I}$ is a dual family for $\F$ which is also a $J$-frame for $\HH$. 

Resuming the example proposed above, assume that $\G=\{g_i\}_{i\in I}$ is a dual family for $\F$. It can be shown that if the subspace $\N_+$ spanned by the predominantly high-frequency signals in $\G$ does not intersect the subspace $\N_-$ spanned by the predominantly low-frequency signals in $\G$ (plus some condition related with the redundancy of $\G$) then $\G$ is also a $J$-frame for $\HH$. In this case, the sampling-reconstruction scheme reconstructs the positive vectors in $\M_-^{\ort}$ in terms of the positive vectors in $\G$, and the negative vectors in $\M_+^{\ort}$ in terms of the negative vectors in $\G$.

Also, Parseval $J$-frames are considered as those $J$-frames which are self-dual in the sense of \eqref{dual2}. In this case, the sampling-reconstruction scheme resembles the decomposition formula provided by an orthonormal basis in a Krein space, see \cite[Ch. 1, \S 10]{AI89}.
\bigskip

The paper is organized as follows. Section \ref{Prelim} contains the notation used along this work and some previously known results, both in frame theory for Hilbert spaces and also on the basic geometry of Krein spaces.

The definition and basic characterizations of $J$-frames are recalled in Section \ref{Jframes}. The first paragraphs consists of an exposition on concepts and notations related with $J$-frames, as the associated decomposition of the coefficients space $\ell_2(I)$, its synthesis and $J$-frame operators, etc. However, some of the material in this section is original, for instance, the contents of Subsection \ref{Jpositive}. 

Section \ref{Duales} is devoted to define and characterize dual families for a given $J$-frame. In particular, if $\F$ is a $J$-frame for a Krein space $\HH$ and $\G$ is a dual family for $\F$, necessary and sufficient conditions are presented in order to determine if $\G$ is also a $J$-frame for $\HH$.

Finally, tight and Parseval $J$-frames are studied in Section \ref{AjustadosParsevals}. For a $J$-frame $\F$ for a Krein space $\HH$, a canonical Parseval $J$-frame $\mc{P}$ is constructed using the square-root of the $J$-frame operator $S$, see \cite[Thm. 6.1]{GLLMMT17}. Also, Parseval $J$-frames for a Krein space $\HH$ are characterized as projections of $J$-orthonormal bases of a larger Krein space $\KK$, using $J$-selfadjoint projections. 
This is a version of Naimark's Theorem for Parseval $J$-frames in Krein spaces.

\section{Preliminaries}\label{Prelim}

If $\HH$ and $\KK$ are Hilbert spaces, $L(\HH,\KK)$ denotes the vector space of bounded linear operators from $\HH$ into $\KK$, and $L(\HH)$ denotes the Banach algebra of bounded linear operators acting on $\HH$.

Given two closed subspaces $\M$ and $\N$ of $\HH$, hereafter $\M\dotplus \N$ denotes their direct sum. In case that $\HH=\M \dotplus\N$, then $P_{\M//\N}$ denotes the (unique) projection with range $\M$ and nullspace $\N$. On the other hand, if $\N\bot \M$ then $\M\oplus \N$ denotes their orthogonal direct sum. In particular, if $\N=\M^\bot$ then $P_\M=P_{\M//\M^\bot}$ denotes the orthogonal projection onto $\M$.

\subsection{Frames for Hilbert spaces}\label{Hilbert}

\noindent The following is the standard notation and some basic results on frames for Hilbert spaces, see \cite{Eldar15,CKP,TaF, Chr, HanLarson}.

\medskip

A \emph{frame} for a Hilbert space $\HH$ is a family of vectors $\F=\{f_i\}_{i\in I}\subset \HH$ for which there exist constants $0<\alpha\leq \beta<\infty$ such that
\begin{equation}\label{ecu frames}
\alpha\ \|f\|^2 \leq \sum_{i\in I} |\langle f,f_i\rangle |^2 \leq \beta\ \|f\|^2\, , 
\end{equation}
for every $f\in \HH$. The optimal constants (maximal for $\alpha$ and minimal for $\beta$) are known, respectively, as the lower and upper frame bounds.

If a family of vectors $\F=\{f_i\}_{i\in I}$ satisfies the upper bound condition in \eqref{ecu frames}, then $\F$ is a \emph{Bessel family}. For a Bessel family $\F=\{f_i\}_{i\in I}$, the \emph{synthesis operator} $T\in L(\ell_2(I),\HH)$ is defined by
\begin{equation}\label{sintesis}
Tx=\sum_{i\in I}\PI{x}{e_i}f_i, \ \ \ x\in\ell_2(I),
\end{equation}
where $\{e_i\}_{i\in I}$ is the standard basis of $\ell_2(I)$. A Bessel family $\F$ is a frame for $\HH$ if and only if $T$ is surjective. In this case, the positive invertible operator $S=TT^*\in L(\HH)$ is called the \emph{frame operator}. It can be easily verified that
\begin{equation}\label{ecu S}
Sf=\sum_{i\in I}\PI{f}{f_i}f_i, \ \ \ f\in\HH.
\end{equation}
This implies that the frame bounds can be computed as: $\alpha=\|S^{-1}\|^{-1}$ and $\beta=\|S\|$.

A Bessel family $\G=\{g_i\}_{i\in I}$ is a \emph{dual} for $\F$ if, for every $f\in\HH$,
\begin{equation}\label{dualidad2}
  f = \sum_{i\in I}\PI{f}{f_i}\, g_i=\sum_{i\in I}\PI{f}{g_i}f_i.
\end{equation}
Observe that if $\G$ is a dual family for $\F$ then it is automatically a frame for $\HH$, because the above equation says that its synthesis operator is also surjective. 

From \eqref{ecu S}, it is also easy to obtain the \emph{canonical reconstruction formula} for the vectors in $\HH$:
\begin{equation}\label{reconst}
f=\sum_{i\in I}\PI{f}{S^{-1}f_i}f_i= \sum_{i\in I}\PI{f}{f_i}S^{-1}f_i, 
\end{equation}
for every $f\in\HH$. The family $\{S^{-1}f_i\}_{i\in I}$ is called the \emph{canonical dual frame} of $\F$. Observe that the canonical dual frame is similar to the original frame $\F$ in the following sense:

\begin{Def}\label{similarity}
Given two frames $\F=\{f_i\}_{i\in I}$ and $\G=\{g_i\}_{i\in I}$ for a Hilbert space $\HH$, they are \emph{similar} if there exists a invertible operator $V\in L(\HH)$ such that $Vf_i=g_i$ for $i\in I$. 
\end{Def}

\subsection{Krein spaces}\label{Krein}

In what follows we present the standard notation and some basic results on Krein spaces. For a complete exposition on the subject (and the proofs of the results below) see \cite{Bognar, AI89, Ando, Dritschel 1}.
\medskip

A vector space $\mathcal H$ with a hermitian sesquilinear form $\K{\,\,}{\,}$ is called a \emph{Krein space} if there exists a so-called
\emph{fundamental decomposition}
\begin{equation}\label{fundamental}
\mathcal H= \mathcal H_+  \sdo\, \mathcal H_-,
\end{equation}
which is the direct (and orthogonal with respect to $\Skindef$) sum of two
Hilbert spaces $(\mathcal H_+, \Skindef)$ and $(\mathcal H_-, -\Skindef)$.
These two Hilbert spaces induce in a natural way a Hilbert space inner product $\Skdef$
and, hence, a Hilbert space topology  on $\mathcal H$.
Observe that the indefinite inner-product $\Skindef$ and the
Hilbert space inner product $\Skdef$ of $\mathcal H$ are related by means of a
\emph{fundamental symmetry}, i.e.\ a unitary self-adjoint operator $J\in L(\mathcal H)$
which satisfies
\[
  \langle x, y\rangle = [Jx, y], \ \  \text{for $x,y\in \mathcal H$}.
\]
Although the fundamental decomposition is not unique, the norms induced by different fundamental decompositions turn out to be equivalent. Therefore, the (Hilbert space) topology in $\HH$ does not depend on the chosen fundamental decomposition.

 \medskip
A vector $x\in \HH$ is \emph{$J$-positive} if $\K{x}{x}>0$. A subspace $\St$ of $\HH$ is \emph{$J$-positive} if every $x\in\St$, $x\neq 0$, is a $J$-positive vector. $J$-nonnegative, $J$-neutral, $J$-negative and $J$-nonpositive vectors (and subspaces) are defined analogously. 
A subspace $\St$ of $\HH$ is said to be \emph{ uniformly $J$-positive} if there exists $\alpha> 0$ such that
\[
\K{x}{x} \geq \alpha \|x\|^2, \  \ \text{for every $x\in\St$},
\]
where $\|\, \|$ stands for the norm of the associated Hilbert space $(\HH,\PI{\,}{\,})$. Uniformly $J$-negative subspaces are defined analogously.

Given a subspace $\St$ of a Krein space $\HH$, the \emph{$J$-orthogonal companion} to $\St$ is defined by 
\[
\St^{\ort}=\{ x\in\HH : \K{x}{s}=0 \; \text{for every $s\in\St$}\}.
\]
A (closed) subspace $\St$ of $\HH$ is \textit{regular} if 
\[
\HH=\St \dotplus \St^{\ort}. 
\]
Equivalently, $\St$ is regular if and only if there exists a (unique) $J$-selfadjoint projection $E$ onto $\St$, see e.g. \cite[Ch.1, Thm. 7.16]{AI89}. In particular, every closed uniformly $J$-definite subspace $\St$ of $\HH$ is a regular subspace of $\HH$.

If $\mathcal H$ and $\mathcal K$ are Krein spaces and
 $T\in L(\mathcal H, \mathcal K)$, the \emph{adjoint operator of $T$ in the sense of Krein spaces} (shortly, the $J$-adjoint of $T$) is the unique operator $T^+\in L(\KK,\HH)$ satisfying
 \[
 [Tx,y]=[x,T^+y], \ \ \text{for every $x\in\HH$, $y\in\KK$}.
 \]
 
 An operator $T\in L(\HH)$ is \emph{selfadjoint in the sense of Krein spaces} (shortly, $J$-selfadjoint) if $T = T^+$, and it is \emph{positive in the sense of Krein spaces} (shortly, $J$-positive) if 
 \[
 \K{Tx}{x}\geq 0, \ \ \text{for every $x\in\HH$}.
 \]
 If $(\HH, \K{\,}{\,})$ is a Krein space, the cone of $J$-positive operators acting on $\HH$ induces a natural order in the real vector space of $J$-selfadjoint operators: if $A$ and $B$ are $J$-selfadjoint operators, $A\leq_J B$ if $B-A$ is a $J$-positive operator. In particular, $0\leq_J B$ means that $B$ is $J$-positive.

\section{$J$-frames: Definition and basic facts}\label{Jframes}

Recently, there were various attempts to introduce frame theory on Krein spaces;
see \cite{EFW,GMMPM12,PW}.
In the following we recall the notion of $J$-frame introduced in \cite{GMMPM12}. Some particular classes of $J$-frames where also considered in \cite{HKP16}.

\medskip

Let $(\HH,\Skindef)$ be a Krein space. Given a Bessel family $\mc{F}=\{f_i\}_{i\in I}$ in $\HH$, set $I_+=\{i\in I: \ \K{f_i}{f_i}\geq 0\}$ and $I_-=\{i\in I: \ \K{f_i}{f_i}< 0\}$. Then, consider the orthogonal decomposition of $\ell_2(I)$ induced by the partition of $I$:
\begin{equation}\label{desc fund}
\ell_2(I)=\ell_2(I_+)\oplus \ell_2(I_-),
\end{equation}
and denote $P_\pm$ the orthogonal projection onto $\ell_2(I_\pm)$, respectively. Also,  if $T:\ell_2(I)\ra\HH$ is the synthesis operator of $\F$, set $T_\pm=TP_\pm$, i.e. 
\begin{equation}\label{tes}
  T_\pm x = \sum_{i\in I_\pm}\PI{x}{e_i}f_i, \qquad x\in\ell_2(I),
\end{equation}
where $\{e_i\}_{i\in I}$ stands for the standard basis of $\ell_2(I)$. The ranges of $T_+$ and $T_-$ are the main ingredient in the following definition of frames for Krein spaces. So, if
\begin{equation}\label{emes}
\M_\pm:= \ol{\Span\{f_i:\ i\in I_\pm\}},
\end{equation}
it is important to recall that $\Span\{f_i:\ i\in I_\pm\}\subseteq R(T_\pm)\subseteq \M_\pm$ and
$R(T)=R(T_+) + R(T_-)$.

\begin{Def}\label{Jframe}
Let $\F=\{f_i\}_{i\in I}$ be a Bessel family in a Krein space $\HH$ and $T_\pm$ be as in \eqref{tes}. Then,
$\mc{F}$ is a \emph{$J$-frame} for $\HH$ if $R(T_+)$ is a maximal uniformly $J$-positive subspace and $R(T_-)$ is a maximal uniformly $J$-negative subspace of $\HH$.
\end{Def}

If $\F$ is a $J$-frame for $\HH$ then $R(T_\pm)=\M_\pm$ and,
\begin{equation}\label{suma}
R(T)=R(T_+) \dotplus R(T_-)=\M_+ \dotplus \M_-=\HH,
\end{equation}
where the last equality follows from \cite[Ch. 1, Corollary 5.2]{AI89}. Thus, $\F$ is also a frame for the Hilbert space $(\HH, \Skdef)$. Also, if $\F_\pm:= \{f_i\}_{i\in I_\pm}$ it is easy to see that
$\F_\pm$ is a frame for the Hilbert space $(\M_\pm, \pm\Skindef)$.
The following is a characterization of a $J$-frame in terms of frame inequalities, see \cite[Theorem 3.9]{GMMPM12}:

\begin{thm}\label{thm J frame bounds}
Let $\F=\{f_i\}_{i\in I}$ be a frame for $\HH$. Then, $\F$ is a $J$-frame if and only if $\M_\pm$  (defined as in \eqref{emes}) are non-degenerated subspaces of $\HH$ and there exist constants $0<\alpha_\pm \leq \beta_\pm$ such that
\begin{equation}\label{eq J frame bounds}
	\alpha_\pm (\pm\K{f}{f}) \leq \sum_{i\in I_\pm} |\K{f}{f_i}|^2 \leq \beta_\pm (\pm\K{f}{f}),
\end{equation}
for every $f\in \M_\pm$.
\end{thm}

If $\F=\{f_i\}_{i\in I}$ is a $J$-frame for $\HH$, consider the projection 
\begin{equation}\label{Q}
Q:=P_{\M_+//\M_-}. 
\end{equation}
It is immediate that the synthesis operator $T$ intertwines the projections $Q$ and $P_+$, i.e.
\begin{equation}\label{intertwin}
QT=TP_+.
\end{equation}
Equivalently, $(I-Q)T=TP_-$.


From now on, given a $J$-frame $\F=\{f_i\}_{i\in I}$ for $\HH$, endow the coefficient space $\ell_2(I)$ with the following indefinite inner-product:
\begin{equation}\label{l2}
\K{x}{y}_2:=\sum_{i\in I_+} x_i \ol{y_i}- \sum_{i\in I_-} x_i \ol{y_i}, 
\end{equation}
where $x=(x_i)_{i\in I}, y=(y_i)_{i\in I} \in \ell_2(I)$.
Then, $(\ell_2(I),\Skindef_2)$ is a Krein space and \eqref{desc fund} is a fundamental decomposition of $\ell_2(I)$.

Observe that $J_2:=P_+ - P_-$ is the fundamental symmetry that relates the indefinite inner-product defined in \eqref{l2} and the standard inner-product of $\ell_2(I)$.

\medskip

The nullspace $N(T)$ of the synthesis operator $T$ of a frame $\F$ for $\HH$ provides relevant information about the frame itself. Indeed, the cardinal number $\dim N(T)$ is called the \emph{excess} of the frame. It measures the amount of vectors that can be removed from the original frame $\F$ without loosing the condition of being a frame for $\HH$, see \cite{ACRS,BCHL,H94}. In the case of a $J$-frame $\F$ for a Krein space $\HH$, the nullspace of the synthesis operator is a regular subspace of $\ell_2(I)$ which admits a natural decomposition in terms of \eqref{desc fund}.

\begin{lema}\label{Jframe descompone el nucleo}
If $\F=\{f_i\}_{i\in I}$ is a $J$-frame for $\HH$ with synthesis operator $T:\ell_2(I)\ra \HH$ then 
\[
P_+P_{N(T)}=P_{N(T)}P_+ \ \text{and} \ P_+P_{N(T)^{\ort}}=P_{N(T)^{\ort}}P_+, 
\]
where $P_{N(T)}$ and $P_{N(T)^{\ort}}$ stand for the orthogonal projections onto $N(T)$ and $N(T)^{\ort}$, respectively.

In particular, $N(T)$ is a regular subspace of $\ell_2(I)$ and the following decompositions hold:
\[
N(T)=N(T)\cap\ell_2(I_+)\ \sdo\ N(T)\cap \ell_2(I_-);
\]
\[
N(T)^{\ort}=N(T)^{\ort}\cap\ell_2(I_+)\ \sdo\ N(T)^{\ort}\cap \ell_2(I_-).
\]
\end{lema}

\begin{proof}
Observe that $Tx=0$ if and only if $T_+x=T_-x=0$ because $R(T_+)\cap R(T_-)=\{0\}$. Thus, $x\in N(T)$ if and only if $P_+x\in N(T)\cap\ell_2(I_+)$ (and $P_-x\in N(T)\cap\ell_2(I_-)$). Hence, $(I-P_{N(T)})P_+P_{N(T)}=0$, and
\begin{eqnarray*}
P_{N(T)}P_+ =(P_+P_{N(T)})^*=(P_{N(T)}P_+P_{N(T)})^* =P_{N(T)}P_+P_{N(T)}=P_+P_{N(T)}.
\end{eqnarray*}
Then, $P_{N(T)}P_+=P_{\ell_2(I_+)\cap N(T)}$ and $P_{N(T)}P_-=P_{\ell_2(I_-)\cap N(T)}$. So, the decomposition
\[
N(T)=N(T)\cap\ell_2(I_+)\ \sdo\ N(T)\cap \ell_2(I_-)
\]
follows. In particular, it shows that $N(T)$ is a regular subspace of $\ell_2(I)$. 

Furthermore, since $P_{N(T)^{\ort}}=J_2(I-P_{N(T)})J_2$ and $P_+$ commutes with $J_2$ and with $P_{N(T)}$, it follows that $P_+P_{N(T)^{\ort}}=P_{N(T)^{\ort}}P_+$ and
\[
N(T)^{\ort}=N(T)^{\ort}\cap\ell_2(I_+)\ \sdo\ N(T)^{\ort}\cap \ell_2(I_-).
\]
So, the proof is completed.
\end{proof}

Let $\F=\{f_i\}_{i\in I}$ be a $J$-frame for $\HH$. If $T:\ell_2(I)\ra\HH$ is the synthesis operator of $\F$, its $J$-adjoint is given by
\begin{equation}
T^+f=\sum_{i\in I} \sigma_i\K{f}{f_i}e_i, \ \ f\in\HH,
\end{equation}
where $\s_i=\sgn\K{f_i}{f_i}$ for $i\in I$.

\begin{Def}
Given a $J$-frame $\F=\{f_i\}_{i\in I}$ for $\HH$, its \emph{$J$-frame operator} $S:\HH\ra\HH$ is defined by
\begin{equation}
Sf:= TT^+f=\sum_{i\in I}\sigma_i \K{f}{f_i}f_i, \ \ f\in\HH.
\end{equation}
\end{Def}

It is easy to see that $S$ is an invertible $J$-selfadjoint operator in the Krein space $\mathcal H$. It provides an indefinite reconstruction formula for the vectors in $\HH$ in terms of $\F$: if $\s_i=\sgn\K{f_i}{f_i}$ then,
\begin{equation}\label{irf}
f=\sum_{i\in I}\s_i \K{f}{S^{-1}f_i} f_i= \sum_{i\in I}\s_i \K{f}{f_i} S^{-1}f_i, 
\end{equation}
for every $f\in\HH$.
In \cite[Proposition 5.4]{GMMPM12} it was shown that the family $\F'=\{S^{-1}f_i\}_{i\in I}$ is also a $J$-frame for $\HH$. 
Moreover, 
\begin{prop}\label{canonical dual}
If $\F=\{f_i\}_{i\in I}$ is a $J$-frame for $\HH$ with $J$-frame operator $S$, then $\F'=\{S^{-1}f_i\}_{i\in I}$ is a $J$-frame for $\HH$ too. Also, for every $i\in I$
\[
\sgn(\K{S^{-1}f_i}{S^{-1}f_i})=\sgn(\K{f_i}{f_i}), 
\]
the $J$-frame operator of $\F'$ is $S^{-1}$ and, if $\M_\pm$ are given by \eqref{emes} then
\begin{equation}\label{emes dual}
\ol{\Span\{S^{-1}f_i:\ i\in I_\pm \}}=\M_\mp^{\ort}.
\end{equation}
\end{prop}

The $J$-frame operator $S$ also interacts in a particular manner with the projection $Q= P_{\M_+//\M_-}$: 
\begin{equation}\label{eses con q}
QS=SQ^+.
\end{equation}
Due to the invertibility of $S$, it also holds that 
\[
S^{-1}Q=Q^+S^{-1}. 
\]
Given a $J$-frame $\F$ for $\HH$ with synthesis operator $T$, \eqref{eses con q} allows to represent the $J$-selfadjoint projection onto $N(T)^{\ort}$ in terms of $T$. 
\begin{lema}\label{desc l2+}
If $\F=\{f_i\}_{i\in I}$ is a $J$-frame for $\HH$ with synthesis operator $T:\ell_2(I)\ra \HH$ then $E=T^+(TT^+)^{-1}T$ is the $J$-selfadjoint projection onto $N(T)^{\ort}$. Moreover, $EP_\pm=P_\pm E$. 
\end{lema}

\begin{proof}
Assume that $\F$ is a $J$-frame for $\HH$ with $J$-frame operator $S=TT^+$. The identity $(S^{-1}T)T^+=S^{-1}S=I$ implies that $E:=T^+ S^{-1}T$ is a projection. Also, it is immediate that $E$ is $J$-selfadjoint and $R(E)=R(T^+)=N(T)^{\ort}$ because $T$ is surjective. Hence, $E$ is the $J$-selfadjoint projection onto $N(T)^{\ort}$.
Moreover, by \eqref{intertwin},  $QT=TP_+$ and
\begin{eqnarray*}
EP_+ = T^+ S^{-1}TP_+=T^+ S^{-1}QT = T^+ Q^+S^{-1}T= P_+T^+ S^{-1}T=P_+E.
\end{eqnarray*}
Then, $EP_+$ is the $J$-selfadjoint projection onto $\ell_2(I_+)\cap N(T)^{\ort}$, see \cite[Prop. 4]{Hassi} (or \cite[Lemma 2.2]{Ando09}).
\end{proof}

\subsection{$J$-positive operators associated to $J$-frame operators}\label{Jpositive}

Given a $J$-frame $\F$ for $\HH$, the operator $S_\pm:\HH\ra\HH$ defined by
\begin{equation}\label{eses}
S_\pm f:= \sum_{i\in I_\pm} \K{f}{f_i}f_i, \ \ f\in\HH,
\end{equation}
is a positive operator in $\HH$ because
\begin{equation}\label{eses positivos}
\K{S_\pm f}{f}=\sum_{i\in I_\pm}|\K{f}{f_i}|^2, 
\end{equation}
for every $f\in \HH$. Thus, the $J$-frame operator $S=S_+-S_-$ is the difference of two $J$-positive operators.
By \eqref{eses con q}, it follows that $S_+=QS=SQ^+$ and $S_-=-(I-Q)S=-S(I-Q)^+$. Therefore,
$R(S_\pm)=\M_\pm$ and
\begin{equation}\label{mapping}
S(\M_-^{\ort})=\M_+ \ \ \text{and} \ \ S(\M_+^{\ort})=\M_-.
\end{equation}

The following theorem shows that $S_+$ is the smallest $J$-selfadjoint operator above $S$ (according to $\leq_J$) whose range is contained in $\M_+$, and $-S_-$ is the biggest $J$-selfadjoint operator below $S$ (according to $\leq_J$) whose range is contained in $\M_-$.
\medskip

\begin{thm}
Let $S:\HH\ra\HH$ be the $J$-frame operator associated to a $J$-frame $\F=\{f_i\}_{i\in I}$. If $S_\pm$ and $\M_\pm$ are given by \eqref{eses} and \eqref{emes}, respectively, then 
\[
S_+=\min_{\leq_J}\{X=X^+:\ S\leq_J X, \ R(X)\subseteq \M_+\};
\]
\[
-S_-=\max_{\leq_J}\{Y=Y^+:\ Y\leq_J S, \ R(Y)\subseteq \M_-\}.
\]
\end{thm}

\begin{proof}
Observe that $P_+$ is a $J_2$-expansive projection (i.e. $I\leq_{J_2} P_+$) because it is $J_2$-selfadjoint and its nullspace is uniformly $J_2$-negative.  Also, $P_-$ is a $J_2$-contractive projection (i.e. $P_-\leq_{J_2} I$) because it is $J_2$-selfadjoint and its nullspace is uniformly $J_2$-positive,  see \cite[Prop. 5]{Hassi}.

First, note that if 
\[
\mathfrak{X}:=\{X=X^+:\ S\leq_J X, \ R(X)\subseteq \M_+\}
\]
 then $S_+\in \mathfrak{X}$. Indeed, $R(S_+)=\M_+$ and, since $P_+$ is $J$-expansive, it follows that $S=TT^+\leq_J TP_+T^+=S_+$. If $X\in \mathfrak{X}$ then $S\leq_J X$ and $R(X)\subseteq \M_+$. So, $\M_+^{\ort}\subseteq R(X)^{\ort}=N(X)$ and, if $Q=P_{\M_+//\M_-}$,  
\[
0 \leq_J Q(X-S)Q^+ = QXQ^+ - QSQ^+=X-S_+.
\]
Thus, $S_+=\min_{\leq_J} \mathfrak{X}$. 

%
On the other hand, if 
\[
\mathfrak{Y}:=\{Y=Y^+:\ Y\leq_J S, \ R(X)\subseteq \M_-\}, 
\]
then $-S_-\in \mathfrak{Y}$. Indeed, $R(S_-)=\M_-$ and, since $P_-$ is $J_2$-contractive, it follows that $-S_-=TP_-T^+\leq_J TT^+=S$. If $Y\in \mathfrak{Y}$ then $Y\leq_J S$ and $R(Y)\subseteq \M_-$. So, $\M_-^{\ort}\subseteq N(Y)$ and
\[
0 \leq_J (I-Q)(S-Y)(I-Q)^+=-S_- - Y.
\]
i.e. $-S_-=\max_{\leq_J} \mathfrak{Y}$. 
\end{proof}

Finally, the class of $J$-frame operators can be characterized in the following way, see \cite[Proposition~5.7]{GMMPM12}.

\begin{thm}\label{oldie}
A bounded invertible $J$-selfadjoint operator $S$ in a Krein space $\mathcal H$ is a $J$-frame
operator if and only if the following conditions are satisfied:
\begin{enumerate}
\item[\rm(i)] there exists a maximal uniformly $J$-positive subspace
 $\mathcal L_+$ of $\mathcal H$ such that $S(\mathcal L_+)$ is also
 maximal uniformly $J$-positive;
\item[\rm(ii)]  $\K{Sf}{f} \geq 0$ for every $f\in \mathcal L_+$;
\item[\rm(iii)]  $\K{Sg}{g} \leq 0$ for every $g \in (S(\mathcal L_+))^{[\perp]}$.
\end{enumerate}
\end{thm}

Note that, if $\F=\{f_i\}_{i\in I}$ is a $J$-frame for $\HH$ with $J$-frame operator $S$ and $\M_\pm$ are given by \eqref{emes}, then $\M_-^{\ort}$ is a maximal uniformly positive subspace satisfying the conditions (i)--(iii) in Theorem \ref{oldie}. In fact, (i) follows from \eqref{mapping}. On the other hand, since $N(S_-)=R(S_-)^{\ort}=\M_-^{\ort}$, for $f\in \M_-^{\ort}$ we have that
\[
\K{Sf}{f}=\K{S_+f}{f}=\sum_{i\in I_+}|\K{f}{f_i}|^2 \geq 0,
\]
see \eqref{eses positivos}. Analogously, $S\big(\M_-^{\ort}\big)^{\ort}=\M_+^{\ort}=N(S_+)$ and if $g\in \M_+^{\ort}$ then
\[
\K{Sg}{g}=\K{-S_-g}{g}=-\sum_{i\in I_-}|\K{g}{f_i}|^2 \leq 0.
\]
Thus, we have also shown conditions (ii) and (iii).

\begin{obs}\label{desig estr}
Given a $J$-frame $\F$ for $\HH$, consider the operator $S_\pm$ given in \eqref{eses positivos}. Since $S_\pm$ is $J$-positive then, for $f\in \M_\mp^{\ort}$, 
\[
\K{S_\pm f}{f}=0 \qquad \Leftrightarrow \qquad S_\pm f=0.
\]
Therefore, $\K{S f}{f}>0$ for every $f\in \M_-^{\ort}\setminus\{0\}$ and $\K{S f}{f}<0$ for every $f\in \M_+^{\ort}\setminus\{0\}$. 
\end{obs}

\medskip

\section{Dual families of $J$-frames}\label{Duales}

\begin{Def}
Given a $J$-frame $\F=\{f_i\}_{i\in I}$ for $\HH$, a Bessel family $\G=\{g_i\}_{i\in I}$ is a \emph{dual family} for $\F$ if $\s_i:=\sgn\K{f_i}{f_i}=\sgn\K{g_i}{g_i}$ for every $i\in I$ and 
\begin{equation}\label{dual}
	f=\sum_{i\in I}\s_i\K{f}{f_i}g_i=\sum_{i\in I}\s_i\K{f}{g_i}f_i, 
\end{equation}
for every $f\in \HH$.
\end{Def}

Every $J$-frame $\F=\{f_i\}_{i\in I}$ for a Krein space $\HH$ admits dual families, e.g. \eqref{irf} 
shows that $\{S^{-1}f_i\}_{i\in I}$ is a dual family for $\F$. 

\medskip

Assume that  $\F$ is a $J$-frame for $\HH$ and $\M_\pm$ are given by \eqref{emes}. 
If $\G=\{g_i\}_{i\in I}$ is a dual family for $\F$ and $\N_\pm:=\ol{\Span\{g_i : i\in I_\pm\}}$, it follows from \eqref{dual} that
\[
\M_\pm^{\ort}\subseteq \N_\mp, \ \ \ \text{or equivalently,}\ \ \ \N_\mp^{\ort}\subseteq \M_\pm.
\]
Therefore, $\N_+$ (resp. $\N_-$) is a regular subspace of $\HH$, because it contains a maximal uniformly $J$-positive (res. $J$-negative) subspace of $\HH$, see \cite[Ch. 1, Ex. 2 to \S 7]{AI89}. Hence, there exists a uniformly $J$-positive (res. $J$-negative) subspace $\mc{T}_\pm$ of $\HH$ such that
\[
\N_\pm= \M_\mp^{\ort} \sdo \mc{T}_\mp.
\]
Then, $\G$ is a frame for $\HH$ because $\HH=\M_+^{\ort}\dotplus \M_-^{\ort}\subseteq \N_- + \N_+=\ol{\Span\{g_i\}_{i\in I}}$, but it is not necessarily a $J$-frame since $\mc{T}_\pm$ need not to be trivial. 

\medskip
\begin{exa}
Given the standard basis $\{e_1,e_2,e_3\}$ of $\CC^3$, consider the indefinite inner-product
\[
\K{x}{y}=x_1\ol{y_1}+x_2\ol{y_2}-x_3\ol{y_3}, 
\]
where $x=(x_1,x_2,x_3), y=(y_1,y_2,y_3)\in\CC^3$.
It is easy to see that $\F=\{f_1=e_1, f_2=e_2, f_3=e_3, f_4=e_1, f_5=e_2, f_6=e_3\}$
is a $J$-frame for $\CC^3$ with $\M_+=\Span\{e_1,e_2\}$ and $\M_-=\Span\{e_3\}$.

Also, it is immediate that the Bessel family $\mc{G}=\{g_1,\ldots, g_6\}$ composed by the vectors
\[
g_1=(\tfrac{1}{2},\tfrac{1}{2}, -\tfrac{1}{2}), \ \ g_2=(1, \tfrac{1}{2}, -1), \ \ g_3=(0,0,\tfrac{1}{2}), 
\]
\[
g_4=(\tfrac{1}{2},-\tfrac{1}{2},\tfrac{1}{2}), \ \ g_5=(-1,\tfrac{1}{2},1), \ \ g_6=(0,0,\tfrac{1}{2}),
\]
is a dual family for $\F$. But note that $\N_+=\Span\{g_1, g_2, g_3, g_4\}=\CC^3$ because $e_1=g_1 + g_4$, $e_2=g_2 + g_5$ and $e_3=\tfrac{1}{2} g_5-\tfrac{1}{2}g_2 +g_1 +g_4$. Therefore, $\G$ is not a $J$-frame for $\CC^3$.
\end{exa}

By Proposition \ref{canonical dual}, the dual family $\{S^{-1}f_i\}_{i\in I}$ is a $J$-frame for $\HH$ (and it is similar to $\F$). The next result gives a characterization of those dual families of a $J$-frame which are themselves $J$-frames.

\begin{thm}\label{J-frames duales}
Let $\F=\{f_i\}_{i\in I}$ be a $J$-frame for $\HH$. Assume that $\G=\{g_i\}_{i\in I}$ is a dual family for $\F$ and define $\N_\pm=\ol{\Span\{g_i : i\in I_\pm\}}$. Then, $\G$ is a $J$-frame for $\HH$ if and only if $\N_+\cap \N_-=\{0\}$ and
\[
N(V)=N(V)\cap\ell_2(I_+) \oplus N(V)\cap \ell_2(I_-),
\]
where $V:\ell_2(I)\ra \HH$ is the synthesis operator of $\G$.
\end{thm}

\begin{proof}
Assume that $\G=\{g_i\}_{i\in I}$ is a $J$-frame for $\HH$ with synthesis operator $V:\ell_2(I)\ra \HH$. By definition, $R(V_+)=\N_+$ is a maximal uniformly $J$-positive subspace and $R(V_-)=\N_-$ is a maximal uniformly $J$-negative subspace. Thus, $\N_+\cap \N_-=\{0\}$. Also, by Lemma \ref{Jframe descompone el nucleo}, the desired decomposition of $N(V)$ follows. 

Conversely, assume that $\G=\{g_i\}_{i\in I}$ satisfies $N(V)=N(V)\cap\ell_2(I_+) \oplus N(V)\cap \ell_2(I_-)$ and  $\N_+\cap \N_-=\{0\}$. This orthogonal decomposition of the nullspace says that the orthogonal projections $P_{N(V)}$ and $P_+$ commute. 
Also, it implies that $N(V)+\ell_2(I_+)=N(V)\cap \ell_2(I_-) \oplus \ell_2(I_+)$ is a closed subspace. This last condition is equivalent to the closedness of $R(V_\pm)=R(VP_\pm)$, see \cite{Bou,Izu}. Therefore, $R(V_\pm)=\N_\pm$.
 On the other hand, if $\N_+\cap \N_-=\{0\}$ then
\begin{eqnarray*}
\HH &=& \N_+ \dotplus \N_-= (\M_+^{\ort} \sdo \mc{T}_-)\dotplus ((\M_-^{\ort} \sdo \mc{T}_+) = (\M_+^{\ort}\dotplus \M_-^{\ort}) \dotplus (\mc{T}_- \dotplus \mc{T}_+) \\ &=& \HH \dotplus (\mc{T}_- \dotplus \mc{T}_+), 
\end{eqnarray*}
which implies that $\mc{T}_-=\mc{T}_+=\{0\}$. Therefore, $\N_\pm= \M_\mp^{\ort}$ and $\G$ is a $J$-frame for $\HH$.
\end{proof}

As a consequence of the proof of Theorem \ref{J-frames duales}, the following is immediate. 

\begin{cor}\label{rangos duales}
Given a $J$-frame $\F$ for $\HH$, consider $\M_\pm$ as in \eqref{emes}. Assume that $\G$ is a dual family for $\F$, with synthesis operator $V$. Then, $\G$ is a $J$-frame if and only if $R(V_\pm)=\M_\mp^{\ort}$.
\end{cor}

Assume that $\F=\{f_i\}_{i\in I}$ is a $J$-frame for $\HH$ and $\G=\{g_i\}_{i\in I}$ is a Bessel family in $\HH$. Let $Q$ be as in \eqref{Q}. If $T$ and $V$ are the synthesis operators of $\F$ and $\G$, respectively, they satisfy Eq. \eqref{dual} if and only if
\[
TP_+V^+=T_+V_+^+=Q,
\]
and $TP_-V^+=T_-V_-^+=I-Q$,
or equivalently, $VP_+T^+=Q^+$ and $VP_-T^+=I-Q^+$.

\medskip 

The following proposition describes the synthesis operators of the dual families associated to a given $J$-frame.

\begin{prop}\label{familiasduales}
Let $\F=\{f_i\}_{i\in I}$ be a $J$-frame for $\HH$ with synthesis operator $T$. Given a Bessel family $\G=\{g_i\}_{i\in I}$ with synthesis operator $V$, $\G$ is a dual family for $\F$ if and only if $V=S^{-1}T + W$ where $W$ satisfies $N(T)^{\ort}\subseteq N(W)$.
\end{prop}

\begin{proof}
Let $\F=\{f_i\}_{i\in I}$ be a $J$-frame for $\HH$ with synthesis operator $T:\ell_2(I)\ra \HH$. By Proposition \ref{canonical dual}, if $S=TT^+$ then $\F'=\{S^{-1}f_i\}_{i\in I}$ is a dual family for $\F$ with synthesis operator $S^{-1}T$.

Hence, if $\G$ is a dual family for $\F$ with synthesis operator $V:\ell_2(I)\ra \HH$, it follows that 
\[
(V-S^{-1}T)T^+=I-I=0. 
\]
Therefore, $V=S^{-1}T + W$ where $W:=V-S^{-1}T$ satisfies $N(T)^{\ort}=R(T^+)\subseteq N(W)$.

Conversely, assume that $V=S^{-1}T + W$ for some $W:\ell_2(I)\ra \HH$ satisfying $N(T)^{\ort}\subseteq N(W)$. Then, it is immediate that $VT^+=S^{-1}TT^+ + WT^+=I$. Moreover, 
by \eqref{intertwin} and \eqref{eses con q},
\begin{eqnarray*}
VP_+T^+ = S^{-1}TP_+T^+ + WP_+T^+ = S^{-1}QTT^+ + WT^+Q^+=  Q^+S^{-1}S =Q^+.
\end{eqnarray*}
So, $TP_+V^+=Q$ and $TP_-V^+=T(I-P_+)V^+=I-Q$, i.e. $\G$ is a dual family for $\F$.
\end{proof}

Let $\F$ be a $J$-frame with synthesis operator $T$. By Lemma \ref{desc l2+}, if $E$ is the $J$-selfadjoint projection onto $N(T)^{\ort}$ and $P_+$ is the orthogonal projection onto $\ell_2(I_+)$, then $EP_+=P_+E$. Therefore, $F_+:=EP_+$ and $F_-:=(I-E)P_+$ are the $J$-selfadjoint projections onto $\ell_2(I_+)\cap N(T)^{\ort}$ and $\ell_2(I_+)\cap N(T)$, respectively. 

If $\G$ is a dual family for $\F$ with synthesis operator $V$, then 
\[
R(V_+)=R(VF_+ + VF_-)=R(VF_+) + R(VF_-).
\]
By Proposition \ref{familiasduales}, $V=S^{-1}T + W$ for some $W$ such that $WT^+=0$. Hence,
\begin{eqnarray*}
VF_+ = (S^{-1}T + W)F_+ = S^{-1}TF_+=S^{-1}TP_+ = Q^+S^{-1}T.
\end{eqnarray*}
On the other hand, $VF_-= (S^{-1}T + W)F_-= WF_-$. Thus,
\begin{eqnarray*}
\N_+ \supseteq R(V_+)= R(VF_+) + R(VF_-) = R(Q^+S^{-1}T) + R(WF_-) =\M_-^{\ort} + W(\ell_2(I_+)\cap N(T)).
\end{eqnarray*}
Analogously, if $V_-:=VP_-$ it follows that 
\[
\N_-\supseteq R(V_-)= \M_+^{\ort} + W(\ell_2(I_-)\cap N(T)).
\]

\begin{cor}
Let $\F$ be a $J$-frame for $\HH$ with synthesis operator $T$. Assume that $\G$ is a dual family for $\F$ with synthesis operator $V=S^{-1}T + W$, where $W$ is such that $WT^+=0$. If 
$\M_\pm$ is given by \eqref{emes} then, $\G$ is a $J$-frame for $\HH$ if and only if $W(\ell_2(I_\pm))\subseteq \M_\mp^{\ort}$.
\end{cor}

\begin{proof}
%
If $\F$ is a $J$-frame for $\HH$ with synthesis operator $T$, Lemma \ref{desc l2+} implies that
\[
\ell_2(I_+)= \ell_2(I_+)\cap N(T) \sdo \ell_2(I_+)\cap N(T)^{\ort}. 
\]
Assume that $\G$ is a dual family for $\F$ with synthesis operator $V=S^{-1}T + W$. Observe that  
%
$WT^+=0$ implies that $W(\ell_2(I_\pm))=W(\ell_2(I_+)\cap N(T))$.

By Corollary \ref{rangos duales}, $\G$ is a $J$-frame for $\HH$ if and only if $R(V_\pm)=\M_\mp^{\ort}$, or equivalently, 
\[
W(\ell_2(I_\pm))=W \ell_2(I_+)\cap N(T))\subseteq \M_\mp^{\ort}. \ \qedhere
\]
%
\end{proof}

Given a $J$-frame $\F=\{f_i\}_{i\in I}$ for $\HH$ and a vector $f\in \HH$, the analysis of $f$ with the canonical dual $J$-frame $\F'=\{S^{-1}f_i\}_{i\in I}$ has minimal $\ell_2$-norm among the families of coefficients obtained by analyzing $f$ with the different dual families for $\F$.

\begin{thm}
Let $\F=\{f_i\}_{i\in I}$ be a $J$-frame for $\HH$. If $\G=\{g_i\}_{i\in I}$ is a dual family for $\F$ then, given $f\in \HH$,
\begin{equation}
	\sum_{i\in I_\pm} |\K{f}{S^{-1}f_i}|^2 \leq \sum_{i\in I_\pm} |\K{f}{g_i}|^2.
\end{equation}
\end{thm}

\begin{proof}
If $T:\ell_2(I)\ra \HH$ is the synthesis operator of $\F$, assume that the synthesis operator of $\G$ is given by $V=S^{-1}T + W$, where $W:\ell_2(I)\ra \HH$ is such that $N(T)^{\ort}\subseteq N(W)$. Observe that 
\begin{eqnarray*}
VP_+V^+ = (S^{-1}T + W)P_+(S^{-1}T + W)^+ = (S^{-1}T)P_+(S^{-1}T)^+ + WP_+W^+,
\end{eqnarray*}
and $WP_+W^+$ is a $J$-positive operator. Hence, 
\[
\K{VP_+V^+ f}{f}\geq \K{(S^{-1}T)P_+(S^{-1}T)^+ f}{f},
\]
for every $f\in \HH$. Furthermore, since $V^+f=\sum_{i\in I}\s_i \K{f}{g_i}e_i$ and $(S^{-1}T)^+f=\sum_{i\in I}\s_i \K{f}{S^{-1}f_i}e_i$, it turns out that
\begin{eqnarray*}
\sum_{i\in I_+} |\K{f}{S^{-1}f_i}|^2 &=& \K{P_+(S^{-1}T)^+f}{(S^{-1}T)^+f} = \K{(S^{-1}T)P_+(S^{-1}T)^+ f}{f}  \\ 
&\leq& \K{VP_+V^+ f}{f} = \K{P_+V^+ f}{V^+f} = \sum_{i\in I_+} |\K{f}{g_i}|^2.
\end{eqnarray*}

Analogously, it is easy to see that $VP_-V^+=(S^{-1}T)P_-(S^{-1}T)^+ + WP_-W^+$, and $WP_-W^+$ is a $J$-negative operator. Hence, 
\[
\K{VP_-V^+ f}{f}\leq \K{(S^{-1}T)P_-(S^{-1}T)^+ f}{f},
\]
for every $f\in \HH$. Therefore,
\begin{eqnarray*}
\sum_{i\in I_-} |\K{f}{S^{-1}f_i}|^2 &=& -\K{P_-(S^{-1}T)^+f}{(S^{-1}T)^+f} = -\K{(S^{-1}T)P_+(S^{-1}T)^+ f}{f}  \\ 
&\leq& -\K{VP_-V^+ f}{f}= -\K{P_-V^+ f}{V^+f} = \sum_{i\in I_-} |\K{f}{g_i}|^2. \qedhere
\end{eqnarray*}
\end{proof}

\section{Tight $J$-frames and Parseval $J$-frames}\label{AjustadosParsevals}

Recently, Hossein et al. defined tight and Parseval $J$-frames for Krein spaces \cite{HKP16}. Their definitions are given in terms of the $J$-frame bounds: given a $J$-frame $\F$ for $\HH$, they say that $\F$ is a $J$-tight frame for $\HH$ if there exists $\alpha_\pm>0$ such that
\[
\alpha_\pm (\pm\K{f}{f})=\sum_{i\in I_\pm}|\K{f}{f_i}|^2,
\]
for every $f\in \M_\pm$, where $\M_\pm$ are given by \eqref{emes}. Moreover, $\F$ is a $J$-Parseval frame for $\HH$ if in addition $\alpha_\pm=1$. 

Also, they associate to $\F$ the real number $\zeta\in [\sqrt{2},2)$ given by $\zeta=c_0(\M_+,\mc{C})+c_0(\M_-,\mc{C})$ where $\mc{C}=\{g\in \HH:\ \K{g}{g}=0\}$ is the cone of $J$-neutral vectors of $\HH$ and $c_0(\M_\pm,\mc{C})$ stands for the cosine of the angle between $\mc{C}$ and the uniformly $J$-definite subspace $\M_\pm$, see \cite[Def. 3.10]{GMMPM12}.


Observe that $\zeta=\sqrt{2}$ if and only if $c_0(\M_+,\mc{C})=c_0(\M_-,\mc{C})=\tfrac{1}{\sqrt{2}}$. Following the same arguments used in the proof of Theorem 3.6 in \cite{HKP16}, it is easy to see that $c_0(\M_\pm,\mc{C})=\tfrac{1}{\sqrt{2}}$ is equivalent to $\M_\pm\subseteq \HH_\pm$, where $\HH=\HH_+\sdo \HH_-$ is the fundamental decomposition that determines the definite inner-product in $\HH$. Moreover, since $\M_\pm$ is a maximal uniformly $J$-definite subspace, it turns out that $c_0(\M_\pm,\mc{C})=\tfrac{1}{\sqrt{2}}$ if and only if $\M_\pm= \HH_\pm$. In particular, $\zeta=\sqrt{2}$ implies the $J$-orthogonality of the subspaces, i.e.  $\M_+\ort\M_-$. 

Thus, the main results in \cite{HKP16} (Theorems 3.6 and 3.10) assume implicitly (or explicitly) the $J$-orthogonality of the maximal uniformly $J$-definite subspaces.

\medskip

Along this section, new definitions for tight and Parseval $J$-frames are presented. These definitions depend on the associated $J$-frame operator. Then, it is shown that they imply the definitions given in \cite{HKP16} and also the $J$-orthogonality of the subspaces $\M_\pm$ defined in \eqref{emes}.

\begin{Def}
A $J$-frame $\mc{F}=\{f_i\}_{i\in I}$ for a Krein space $\HH$ is a \emph{tight $J$-frame} for $\HH$ if its $J$-frame operator coincides with $\alpha I$ for some $\alpha>0$, i.e. if
\[
TT^+=\alpha I, \ \ \ \alpha>0, 
\]
where $T:\ell_2(I)\ra\HH$ is the synthesis operator of $\mc{F}$.
\end{Def}

\begin{prop}\label{Ajustados}
Let $\mc{F}=\{f_i\}_{i\in I}$ be a $J$-frame for $\HH$. Then, $\mc{F}$ is a tight $J$-frame for $\HH$ if and only if $\mc{F}_\pm=\{f_i\}_{i\in I_\pm}$ is a tight frame for the Hilbert space $(\M_\pm,\pm\K{\, \,}{\,})$, the frame bounds of $\F_+$ and $\F_-$ coincide, and $\M_-=\M_+^{\ort}$.
\end{prop}

\begin{proof}
Given a $J$-frame $\mc{F}=\{f_i\}_{i\in I}$ for $\HH$, consider the $J$-positive operators $S_\pm=\pm T_\pm T_\pm^+$. 

If $\mc{F}=\{f_i\}_{i\in I}$ is a tight $J$-frame with frame bound $\alpha>0$ then $S_+ - S_-=S=\alpha I$. Therefore, $S_+=\alpha I + S_-$ and $S_+S_-=S_-S_+$. But, if $S_+S_-=S_-S_+$ then $R(S_+S_-)\subseteq\M_+\cap\M_-=\{0\}$. So,
\[
S_+ - S_-=\alpha I \ \ \ \text{and} \ \ \ S_+ S_-=S_-S_+=0.
\]
Note that $S_-S_+=0$ implies $\M_+=R(S_+)\subseteq N(S_-)=R(S_-)^{\ort}=\M_-^{\ort}$. Since both $\M_+$ and $\M_-^{\ort}$ are maximal uniformly $J$-definite subspaces of $\HH$, it holds that $\M_+=\M_-^{\ort}$. 

Finally, observe that $\mc{F}_+=\{f_i\}_{i\in I_+}$ is a tight frame for $(\M_+,\K{\, \,}{\,})$ because
\[
\sum_{i\in I_+}|\K{f}{f_i}|^2= \sum_{i\in I}|\K{f}{f_i}|^2=\K{S f}{f}=\alpha \K{f}{f},
\]
for every $f\in \M_+$. Analogously, $\mc{F}_-=\{f_i\}_{i\in I_-}$ is a tight frame for $(\M_-,-\K{\, \,}{\,})$ with frame bound $\alpha$.

Conversely, suppose that $\M_-=\M_+^{\ort}$. Then, the projection $Q$ given in \eqref{Q} is the $J$-selfadjoint projection onto $\M_+$.

Also, assume that $\mc{F}_\pm=\{f_i\}_{i\in I_\pm}$ is a tight frame for the Hilbert space $(\M_\pm,\pm\K{\, \,}{\,})$ and the frame bounds of $\F_+$ and $\F_-$ coincide, say both are equal to $\alpha>0$. Hence, $T_+T_+^+=\alpha Q$ and $T_-T_-^+=-\alpha (I-Q)$. Therefore, $TT^+=\alpha I$, i.e. $\mc{F}$ is a tight $J$-frame.
\end{proof}

\begin{Def}
A $J$-frame $\mc{F}=\{f_i\}_{i\in I}$ for a Krein space $\HH$ is a \emph{Parseval $J$-frame} for $\HH$ if its $J$-frame operator coincides with $I$, i.e. if
\[
TT^+=I,
\]
where $T:\ell_2(I)\ra\HH$ is the synthesis operator of $\mc{F}$.
\end{Def}

They are called Parseval $J$-frames because the indefinite reconstruction formula induced by these frames is a signed Parseval's identity, like the reconstruction formula given by an orthonormal $J$-basis in a Krein space \cite[Ch. 1, \S10]{AI89}: given a $J$-frame $\mc{F}=\{f_i\}_{i\in I}$ for $\HH$, $\F$ is a Parseval $J$-frame if and only if
\[
f=\sum_{i\in I} \s_i \K{f}{f_i} f_i, \ \ \ f\in \HH,
\]
where $\s_i=\sgn\K{f_i}{f_i}$.

The following characterization of Parseval $J$-frames is a particular case of Proposition \ref{Ajustados}.

\begin{prop}\label{Parsevalitos}
Let $\mc{F}=\{f_i\}_{i\in I}$ be a $J$-frame for $\HH$. Then, $\mc{F}$ is a Parseval $J$-frame for $\HH$ if and only if $\mc{F}_\pm=\{f_i\}_{i\in I_\pm}$ is a Parseval frame for the Hilbert space $(\M_\pm,\pm\K{\,}{\,})$ and $\M_-=\M_+^{\ort}$. 
\end{prop}


Given Krein spaces $\HH$ and $\KK$, recall that a linear operator $U: \HH\ra \KK$ is a \emph{$J$-isometry} if 
\[
\K{Ux}{Uy}_\KK=\K{x}{y}_\HH \ \ \ \text{for every $x,y\in \HH$}.
\]
On the other hand, $U$ is a \emph{$J$-partial isometry} if $N(U)$ is a regular subspace of $\HH$ and $U$ preserves the indefinite inner product on the subspace $N(U)^{\ort}$, or equivalently,
\[
\K{Ux}{Uy}_\KK=\K{E x}{y}_\HH \ \ \ \text{for every $x,y\in \HH$},
\]
where $E$ is the $J$-selfadjoint projection onto $N(U)^{\ort}$. Furthermore, if $U$ is bounded, its $J$-adjoint $U^+$ is also a $J$-partial isometry.

\begin{cor}\label{sintesis Parseval}
Let $\mc{F}=\{f_i\}_{i\in I}$ be a $J$-frame for $\HH$ with synthesis operator $T$. Then, the following conditions are equivalent:
\begin{enumerate}
	\item $\mc{F}$ is a Parseval J-frame;
	\item $T$ is a $J$-coisometry, i.e. $T^+$ is a $J$-isometry;
	\item $T^+T$ is the $J$-selfadjoint projection onto $N(T)^{\ort}$.
\end{enumerate}
\end{cor}

\subsection{The canonical Parseval $J$-frame associated to a $J$-frame}\label{DualCanon}

Given a frame $\mc{F}=\{f_i\}_{i\in I}$ for a Hilbert space $\HH$, there exists a canonical Parseval frame associated to $\mc{F}$. Indeed,  if $T:\ell_2(I)\ra \HH$ is synthesis operator of $\mc{F}$ then $\mc{P}=\{(TT^*)^{-1/2}f_i\}_{i\in I}$ turns out to be a Parseval frame. 

It is known that the canonical Parseval frame $\mc{P}$ 
is the Parseval frame which is closest to $\F$, with respect to the distance given by
\[
d(\F, \G)= \sum_{i\in I} \|f_i-g_i\|^2, 
\]
where $\mc{F}=\{f_i\}_{i\in I}$ and $\mc{G}=\{f_i\}_{i\in I}$ are frames for $\HH$, see \cite{Balan, CK07}.
Moreover, $\mc{F}$ and $\mc{P}$ are similar frames in the sense of Definition \ref{similarity}. 

The aim of these paragraphs is to show that, given a $J$-frame $\F$ for a Krein space $\HH$, although there are infinitely many Parseval $J$-frames which are similar to $\F$, it is possible to define a ``canonical'' Parseval $J$-frame associated to $\F$.

\medskip

Given a $J$-frame $\F=\{f_i\}_{i\in I}$ for a Krein space $\HH$, consider its $J$-frame operator $S$. 
In \cite[Thm. 5.1]{GLLMMT17} it was proved that the spectrum of $S$ is contained in the right half-plane and, if $\Gamma$ is a Jordan curve enclosing $\sigma(S)$ and contained in the right half-plane, 
\begin{equation}\label{sqrt}
 S^{1/2}:=\frac{1}{2\pi i}\int_\Gamma z^{1/2} (z-S)^{-1} dz,
\end{equation}
defines an invertible self-adjoint square-root for $S$. Moreover, it is the unique (bounded) operator $P$ acting on $\HH$ which satisfies $P^2=S$ and
\[
\sigma(P)\subseteq \big\{z\in\CC: \ z=re^{it}, \ r>0,\ t\in(-\tfrac{\pi}{4},\tfrac{\pi}{4})\ \big\},
\]
see \cite[Thm. 6.1]{GLLMMT17}.

\begin{prop}\label{desc Parseval}
Given a $J$-frame $\F=\{f_i\}_{i\in I}$ for $\HH$ with $J$-frame operator $S$, consider the subspaces $\M_\pm$ given by \eqref{emes}.  
If $S^{1/2}$ is  the self-adjoint square root of $S$ defined in \eqref{sqrt}, let  
\[
\mc{L}_\pm:= S^{1/2}(\M_\mp^{\ort}).
\] 
Then, $\mc{L}_+$ is maximal uniformly $J$-positive, $\mc{L}_-$ is maximal uniformly $J$-negative and
\begin{equation}
	\HH=\mc{L}_+ [\dotplus] \mc{L}_-\ ,
\end{equation}
is a fundamental decomposition of $\HH$.
\end{prop}

\begin{proof}
If $x\in \mc{L}_+$, there exists a (unique) $u\in \M_-^{\ort}$ such that $x=S^{1/2}u$. Then, by Remark \ref{desig estr}, if $x\neq 0$
\[
\K{x}{x}=\K{S^{1/2}u}{S^{1/2}u}=\K{Su}{u}> 0.
\]
Thus, $\mc{L}_+$ is a $J$-positive subspace. Analogously, $\mc{L}_-$ is a $J$-negative subspace and $\mc{L}_+\cap\mc{L}_-=\{0\}$.

Also, if $x=S^{1/2}u\in\mc{L}_+$ and $y=S^{1/2}v\in \mc{L}_-$, where $u\in \M_-^{\ort}$ and $v\in\M_+^{\ort}$, then
\[
\K{x}{y}=\K{S^{1/2}u}{S^{1/2}v}=\K{Su}{v}=0,
\]
because $Su\in S(\M_-^{\ort})=\M_+$ and $v\in\M_+^{\ort}$. Therefore, $\mc{L}_+ \ort \mc{L}_-$. Moreover, $\HH=\M_-^{\ort} \dotplus \M_+^{\ort}$ implies that
\[
\HH=S^{1/2}(\HH)= S^{1/2}(\M_-^{\ort} \dotplus \M_+^{\ort})\subseteq \mc{L}_+ \sdo \mc{L}_- \ .
\]
Then, according to \cite[Ch. 1, Prop. 5.6]{AI89}, $\mc{L}_+$ is a maximal uniformly $J$-positive subspace of $\HH$, $\mc{L}_-$ is a maximal uniformly $J$-negative subspace of $\HH$ and
\[
	\HH=\mc{L}_+ [\dotplus] \mc{L}_-\ . \qedhere
\]
\end{proof}

\begin{thm}\label{Parseval canonico}
Let $\mc{F}=\{f_i\}_{i\in I}$ be a $J$-frame for a Krein space $\HH$ with $J$-frame operator $S$. If $S^{1/2}$ is  the self-adjoint square root of $S$ defined in \eqref{sqrt}, then
\begin{equation}\label{Parseval}
\mc{P}=\{S^{-1/2}f_i\}_{i\in I}
\end{equation}
is a Parseval $J$-frame for $\HH$ which is similar to $\F$.
\end{thm}

\begin{proof}
If $T:\ell_2(I)\ra\HH$ is the synthesis operator associated to $\F$, it is easy to see that the synthesis operator associated to $\mc{P}$ is $S^{-1/2}T$. Hence, $S^{-1/2}$ establishes the similarity between $\F$ and $\mc{P}$.

Also, for each $i\in I_+$ there exists a unique $u_i\in \M_-^{\ort}$ such that $f_i=Su_i$ and, 
\begin{eqnarray*}
\K{S^{-1/2}f_i}{S^{-1/2}f_i} &=& \K{S^{-1/2}Su_i}{S^{-1/2}Su_i} \\ &=& \K{Su_i}{u_i}> 0,
\end{eqnarray*}
see Remark \ref{desig estr}. Analogously, for each $i\in I_-$ there exists a unique $v_i\in \M_-^{\ort}$ such that $f_i=Sv_i$ and 
\begin{eqnarray*}
\K{S^{-1/2}f_i}{S^{-1/2}f_i} &=& \K{S^{-1/2}Sv_i}{S^{-1/2}Sv_i} \\ &=& \K{Sv_i}{v_i}< 0.
\end{eqnarray*}
Hence, $\sgn\K{S^{-1/2}f_i}{S^{-1/2}f_i}=\sgn\K{f_i}{f_i}$ for every $i\in I$.

Then, $R(S^{-1/2}TP_\pm)=S^{-1/2}(R(TP_\pm))=S^{-1/2}(\M_\pm)=\mc{L}_\pm$ is maximal uniformly $J$-definite and $\mc{P}$ is a $J$-frame for $\HH$. Moreover, $\mc{P}$ is a Parseval $J$-frame because 
\begin{eqnarray*}
(S^{-1/2}T)(S^{-1/2}T)^+ &=& S^{-1/2}TT^+S^{-1/2} \\ &=& S^{-1/2}SS^{-1/2}=I. \ \ \ \qedhere
\end{eqnarray*}
\end{proof}

Given a $J$-frame $\F=\{f_i\}_{i\in I}$ for $\HH$, consider the subspaces $\M_\pm$ given by \eqref{emes}.  Theorem \ref{Parseval canonico} establishes that the $J$-positive vectors $\mc{P_+}=\{S^{-1/2}f_i\}_{i\in I_+}$ in the Parseval $J$-frame $\mc{P}$ generate the maximal uniformly $J$-positive subspace $\mc{L}_+$, and the $J$-negative vectors $\mc{P_-}=\{S^{-1/2}f_i\}_{i\in I_-}$ in $\mc{P}$ generate the maximal uniformly $J$-negative subspace $\mc{L}_-=\mc{L}_+^{\ort}$. 

If $U\in L(\HH)$ is a $J$-unitary operator mapping $\mc{L}_+$ onto the $J$-positive subspace $\HH_+$ of a fundamental decomposition $\HH=\HH_+\sdo \HH_-$, it is easy to see that $\mc{G}=\{US^{-1/2}f_i\}_{i\in I}$ is also a Parseval $J$-frame for $\HH$. Moreover, the $J$-positive vectors $\mc{G_+}=\{US^{-1/2}f_i\}_{i\in I_+}$ in  $\mc{G}$ generate the maximal uniformly $J$-positive subspace $U(\mc{L}_+)=\HH_+$, and the $J$-negative vectors $\mc{G_-}=\{US^{-1/2}f_i\}_{i\in I_-}$ in $\mc{G}$ generate the maximal uniformly $J$-negative subspace $U(\mc{L}_-)=\HH_-$. Therefore,

\begin{prop}
Let $\mc{F}=\{f_i\}_{i\in I}$ be a $J$-frame for a Krein space $\HH$ and consider a fixed fundamental decomposition $\HH=\HH_+\sdo \HH_-$. Then, there exists a Parseval $J$-frame $\mc{Q}=\{q_i\}_{i\in I}$ for $\HH$ which is similar to $\F$ and satisfies
\[
\HH_+=\ol{\Span\{q_i: \K{q_i}{q_i}>0\}}. 
\]
\end{prop}

\subsection{Every Parseval $J$-frame is the projection of a $J$-orthonormal basis}\label{Naimark}

A well known result in frame theory for Hilbert spaces is Naimark's Theorem (see e.g. \cite{HanLarson} or \cite[Thm. 1.9]{CKP}). It states that every Parseval frame is the projection of an orthonormal basis. More precisely, $\F=\{f_i\}_{i\in I}$ is a Parseval frame for the Hilbert space $\HH$ if and only if there exist a Hilbert space $\KK$ and an orthonormal basis $\{b_i\}_{i\in I}$ of $\KK$ such that $\HH$ is a (closed) subspace of $\KK$ and, if $P\in L(\KK)$ denotes the orthogonal projection onto $\HH$,
\[
f_i =Pb_i, \ \ \ \text{for every $i\in I$}.
\]
It also admits the following alternative statement:
\begin{thm*}
Let $\F=\{f_i\}_{i\in I}$ be a frame for the Hilbert space $\HH$ with synthesis operator $T: \ell_2(I)\ra \HH$. 
If $\{e_i\}_{i\in I}$ denotes the standard basis of $\ell_2(I)$, the following conditions are equivalent:
\begin{enumerate}
	\item $\F=\{f_i\}_{i\in I}$ is a Parseval frame for $\HH$;
	\item $T^*f_i=P_{N(T)^\bot}e_i$ for every $i\in I$.
\end{enumerate}
\end{thm*}

The following is a natural generalization of the above result to Parseval $J$-frames.

\begin{thm}\label{caso P}
Let $\F=\{f_i\}_{i\in I}$ be a frame for a Krein space $\HH$. 
Then, the following conditions are equivalent: 
\begin{enumerate}
	\item $\F$ is a Parseval $J$-frame for $\HH$;
	\item there exist a Krein space $\KK=\KK_+ \sdo \KK_-$, containing $\HH$ as a regular subspace, and a $\K{\,}{\,}_\KK$-orthonormal basis $\{b_i\}_{i\in I}$ of $\KK$, such that $\{f_i\}_{i\in I_\pm} \subseteq \KK_\pm$ and 
\begin{equation}\label{proy}
f_i =E b_i, \ \ \ \text{for every $i\in I$},
\end{equation}
where  $E\in L(\KK)$ is the $J$-selfadjoint projection onto $\HH$.
\end{enumerate}
\end{thm}

\begin{proof}
By Proposition \ref{Parsevalitos}, if $\F=\{f_i\}_{i\in I}$ is a Parseval $J$-frame for $\HH$ with synthesis operator $T:\ell_2(I)\ra \HH$, then $\M_+$ and $\M_-$ are $J$-orthogonal maximal uniformly $J$-definite subspaces of $\HH$, i.e. $\HH=\M_+\sdo \M_-$ is a fundamental decomposition of $\HH$.

Also, $\F_+=\{f_i\}_{i\in I_+}$ is a Parseval frame for the Hilbert space $(\M_+, \K{\, \,}{\,})$. Then, by Naimark's Theorem, there exist a Hilbert space $\KK_+$ and an orthonormal basis $\{b_i\}_{i\in I_+}$ of $\KK_+$ such that $\M_+$ is a (closed) subspace of $\KK_+$ and, if $E_+\in L(\KK_+)$ denotes the orthogonal projection onto $\M_+$,
\[
f_i =E_+b_i, \ \ \ \text{for every $i\in I_+$}.
\]
Analogously, there exist a Hilbert space $\KK_-$ and an orthonormal basis $\{b_i\}_{i\in I_-}$ of $\KK_-$ such that $\M_-$ is a (closed) subspace of $\KK_-$ and, if $E_-\in L(\KK_-)$ denotes the orthogonal projection onto $\M_-$,
\[
f_i =E_-b_i, \ \ \ \text{for every $i\in I_-$},
\]
because $\F_-=\{f_i\}_{i\in I_-}$ is a Parseval frame for the Hilbert space $(\M_-, -\K{\, \,}{\,})$.	

Now, consider the Hilbert space $\KK=\KK_+ \oplus\KK_-$ and the (fundamental) symmetry $J':\KK\ra\KK$ given by
\[
J'=2 P_{\KK_+//\KK_-} - I=\matriz{I}{0}{0}{-I}.
\]
If $\K{\, \,}{\,}_\KK$ is the indefinite inner-product on $\KK$ induced by $J'$, then $(\KK,\K{\, \,}{\,}_\KK)$ is a Krein space.

Observe that $\{f_i\}_{i\in I_\pm} \subseteq \KK_\pm$ and $\HH$ is a regular subspace of $\KK$ because $\M_\pm$ is a (closed) subspace of $\KK_\pm$ and $\KK_+ [\bot]\, \KK_-$. Furthermore, the $J'$-selfadjoint projection onto $\HH$ is given by $E=\matriz{E_+}{0}{0}{E_-}$.

Finally, note that $\mc{B}=\{b_i\}_{i\in I}=\{b_i\}_{i\in I_+}\cup \{b_i\}_{i\in I_-}$ is a $J'$-orthonormal basis of $\KK$ and
\[
f_i=E b_i, \ \ \ \text{for every $i\in I$}.
\]

Conversely, assume that there exist a Krein space $\KK=\KK_+ \sdo \KK_-$ such that $\{f_i\}_{i\in I_\pm} \subseteq \KK_\pm$, and a $\K{\,}{\,}_\KK$-orthonormal basis $\{b_i\}_{i\in I}$ of $\KK$ satisfying \eqref{proy}.

Then, $\HH=R(T_+)\sdo R(T_-)$ because $R(T_\pm)\subseteq \KK_\pm$. It is also easy to see that $R(T_\pm)$ is a (closed) uniformly $J$-definite subspace of $\HH$, hence it is a fundamental decomposition of $\HH$. Thus, $\F$ is a $J$-frame for $\HH$.

Furthermore, if $E_\pm$ denotes the $J'$-selfadjoint projection onto $R(T_\pm)$, then $E=E_+ + E_-$ (cf. \cite[Prop. 4]{Hassi}, \cite[Thm. 2.3]{Ando09}). 

It remains to show that $\F_\pm$ is a Parseval frame for $(R(T_\pm), \pm\K{\, \,}{\,})$. If $f\in R(T_+)$ then $f=\sum_{i\in I}\s_i \K{f}{b_i}b_i=\sum_{i\in I_+}\K{f}{b_i}b_i$ and
\begin{eqnarray*}
\K{f}{f}  = \sum_{i\in I_+}|\K{f}{b_i}|^2 = \sum_{i\in I_+}|\K{Ef}{b_i}|^2 = \sum_{i\in I_+}|\K{f}{f_i}|^2.
\end{eqnarray*}
Thus, $\F_+=\{f_i\}_{i\in I_+}$ is a Parseval frame for $(R(T_+), \K{\, \,}{\,})$. Analogously, it follows that $\F_-=\{f_i\}_{i\in I_-}$ is a Parseval frame for $(R(T_-), -\K{\, \,}{\,})$. Then, by Proposition \ref{Parsevalitos}, $\F$ is a Parseval $J$-frame for $\HH$.
\end{proof}

In the proof of Theorem \ref{caso P}, the Krein space $\KK$ is constructed in such a way that $\KK_\pm$ is isometrically isomorphic (as Hilbert spaces) to $\ell_2(I_\pm)$ and $\KK$ is isometrically isomorphic (as Krein spaces) to $\ell_2(I)$ (endowed with the fundamental symmetry $J_2$). Therefore, it can be reinterpreted in the following way:

\begin{cor}
Let $\F=\{f_i\}_{i\in I}$ be a frame for a Krein space $\HH$ with synthesis operator $T$.  
Then, the following conditions are equivalent: 
\begin{enumerate}
	\item $\F$ is a Parseval $J$-frame for $\HH$;
	\item $N(T)= N(T)\cap \ell_2(I_+)\, \sdo\, N(T)\cap\ell_2(I_-)$ 
	and
	\[
	T^+ f_i= (I-F) e_i, \ \ \ i\in I,
	\]
	where $F:\ell_2(I)\ra \ell_2(I)$ is the $J$-selfadjoint projection onto $N(T)$.
\end{enumerate}
\end{cor}

\begin{proof}
If $\F$ is a Parseval $J$-frame with synthesis operator $T$, Corollary \ref{sintesis Parseval} implies that $T^+T=I-F$. Then, $T^+f_i=T^+Te_i=(I-F)e_i$ for every $i\in I$. On the other hand, the decomposition for $N(T)$ follows from Lemma \ref{Jframe descompone el nucleo}.

\medskip 

Conversely, assume that $N(T)= N(T)\cap \ell_2(I_+)\, \sdo\, N(T)\cap\ell_2(I_-)$. Then, $N(T)$ is a regular subspace of $\ell_2(I)$ and its associated $J$-selfadjoint projection $F$ satisfies $T^+ f_i= (I-F) e_i$, $i\in I$. Hence, $T^+T=I-F$ and $TT^+$ is also a $J$-selfadjoint projection. Since $R(TT^+)=R(T)=\HH$, it follows that $TT^+=I$. 

It remains to show that $\F$ is a $J$-frame. Note that $F$ commutes with $P_+$ because $N(T)= N(T)\cap \ell_2(I_+)\, \sdo\, N(T)\cap\ell_2(I_-)$. Then, $P_+(I-F)=(I-F)P_+$ is the $J$-selfadjoint projection onto $R(T^+)\cap \ell_2(I_+)$ and $T_+T_+^+$ is a $J$-selfadjoint projection:
\begin{eqnarray*}
(T_+T_+^+)^2 &=& T_+(T_+^+T_+)T_+^+=T_+(T^+T)T_+^+ = TP_+(I-F)P_+T^+\\ &=& T(I-F)P_+T^+ = TP_+T^+=T_+T_+^+.
\end{eqnarray*}
Furthermore, $I=TT^+=T_+T_+^+ + T_-T_-^+$ implies that $T_-T_-^+=I - T_+T_+^+$.

Finally, $T_+= TP_+= T(I-F)P_+= TP_+(I-F)=T_+ (I-F)$ gives that $R(T_+)= R(T_+(I-F))= T_+(R((I-F))=T_+ R(T^+)=R(TP_+ T^+)= R(T_+T_+^+)$. Thus, $R(T_+)= R(T_+T_+^+)$ is uniformly $J$-positive. Analogously, $R(T_-)= R(T_-T_-^+)$ is uniformly $J$-negative and 
\[
\HH=R(T_+)\,\sdo\, R(T_-).
\]
So, $\F$ is a Parseval $J$-frame.
\end{proof}

\noi In the Hilbert spaces setting, the existence of a Parseval dual frame was characterized in \cite[Thm. 2.2]{Han}:

\begin{thm*}
Let $\F=\{f_i\}_{i\in I}$ be a frame for $\HH$. Then, $\F$ admits a Parseval dual frame if and only if there exists a Hilbert $\KK\supset \HH$, an orthonormal basis $\{e_i\}_{i\in I}$ for $\KK$ and an oblique projection $Q\in L(\KK)$ onto $\HH$ such that
\[
Qe_i=f_i, \ \ \ i\in I.
\]
\end{thm*}

Given a $J$-frame $\F=\{f_i\}_{i\in I}$ for $\HH$, assume that $\G=\{g_i\}_{i\in I}$ is a dual $J$-frame for $\F$. 
Theorem \ref{J-frames duales} shows that the subspaces spanned by the $J$-positive (resp. $J$-negative) vectors of these $J$-frames are strongly related:
\[
\N_\pm=\M_\mp^{\ort}.
\]
If it is also assumed that $\G$ is a Parseval $J$-frame, then in particular $\N_-=\N_+^{\ort}$ (cf. Proposition \ref{Parsevalitos}). Thus, this extra assumption imposes that $\F$ generates $J$-orthogonal subspaces:
\[
\M_+^{\ort}=\N_-=\N_+^{\ort}=(\M_-^{\ort})^{\ort}=\M_-.
\]
Then, a necessary condition for the existence of a dual Parseval $J$-frame for $\F$ is that $\M_-=\M_+^{\ort}$.

\begin{thm}
Let $\F=\{f_i\}_{i\in I}$ be a $J$-frame for $\HH$ such that $\K{f_i}{f_j}=0$ for $i\in I_+$, $j\in I_-$. Then, the following statements are equivalent:
\begin{enumerate}
\item $\F$ has a dual family which is a Parseval $J$-frame;
\item there exist a Krein space $\KK=\KK_+ \sdo \KK_-$, containing $\HH$ as a regular subspace, a $\K{\,}{\,}_\KK$-orthonormal basis $\{b_i\}_{i\in I}$ of $\KK$ and an oblique projection $Q\in L(\KK)$ onto $\HH$ such that $\{f_i\}_{i\in I_\pm} \subseteq \KK_\pm$ and
\[
f_i =Q b_i, \ \ \ \text{for every $i\in I$}.
\]
\end{enumerate}
\end{thm}

\begin{proof}
Assume that $\F$ is a $J$-frame such that $\K{f_i}{f_j}=0$ for $i\in I_+$, $i\in I_-$. Then, $\F_\pm=\{f_i\}_{i\in I_\pm}$ is a frame for the Hilbert space $(\M_\pm, \pm\K{\,}{\,})$, where $\M_\pm$ are given by \eqref{emes}. 
Also, the assumption on $\F$ guarantees the $J$-orthogonality between $\M_+$ and $\M_-$.

Then, observe that $\F$ has a dual family which is a Parseval $J$-frame if and only if $\F_\pm$ admits a Parseval dual frame in the Hilbert space $(\M_\pm, \pm\K{\,}{\,})$.

Moreover, by \cite[Thm. 2.2]{Han}, $\F_\pm$ admits a Parseval dual frame if and only if there exists a Hilbert $\KK_\pm\supset \M_\pm$, an orthonormal basis $\{e_i\}_{i\in I_\pm}$ for $\KK_\pm$ and an oblique projection $Q_\pm\in L(\KK_\pm)$ onto $\HH$ such that
\[
Q_\pm e_i=f_i, \ \ \ i\in I_\pm.
\]

Hence, considering the Krein space $\KK$ with fundamental decomposition $\KK=\KK_+\sdo\KK_-$, the family $\{e_i\}_{i\in I}=\{e_i\}_{i\in I_+}\cup \{e_i\}_{i\in I_-}$ is a $J$-orthonormal basis for $\KK$ and the oblique projection $Q\in L(\KK)$ defined by
\[
Qx=Q_+ x_+  + Q_- x_-, \ \ \ \text{if $x=x_+ + x_-$},
\]
where $x_\pm\in \KK_\pm$, satisfies $Qe_i=f_i$ for every $i\in I$ and $R(Q)=R(Q_+)+R(Q_-)=\M_+ + \M_-=\HH$.
\end{proof}

\subsection*{Contact information}

{\bf Juan Ignacio Giribet}

Facultad de Ingenier\'ia, Universidad de Buenos Aires

Av. Paseo Col\'on 850 (C1063ACV) Buenos Aires, Argentina 

and Instituto Argentino de Matem\'atica "Alberto P. Calder\'on" (CONICET) 

Saavedra 15 (C1083ACA) Buenos Aires, Argentina 

jgiribet@fi.uba.ar

\noindent
{\bf Alejandra Maestripieri}

Facultad de Ingenier\'ia, Universidad de Buenos Aires

Av. Paseo Col\'on 850 (C1063ACV) Buenos Aires, Argentina 

and Instituto Argentino de Matem\'atica "Alberto P. Calder\'on" (CONICET) 

Saavedra 15 (C1083ACA) Buenos Aires, Argentina 

amaestri@fi.uba.ar

\noindent
{\bf Francisco Mart\'{\i}nez Per\'{\i}a }

Depto. de Matem\'{a}tica, Fac. Cs. Exactas, Universidad Nacional de La Plata

C.C.\ 172, (1900) La Plata, Argentina

and Instituto Argentino de Matem\'{a}tica "Alberto P. Calder\'{o}n" (CONICET)

Saavedra 15 (C1083ACA) Buenos Aires, Argentina

francisco@mate.unlp.edu.ar

\end{document}